\documentclass{amsart}
\usepackage{mathrsfs}
\usepackage[OT2, T1]{fontenc}
\usepackage{url}
\usepackage{amsmath}
\usepackage{array}
\usepackage{graphicx}
\usepackage{amsfonts}
\usepackage{amssymb}
\usepackage{amstext}
\usepackage{amsthm}
\usepackage[alphabetic]{amsrefs}
\usepackage{hyperref}
\usepackage{colonequals}
\usepackage{enumitem}
\usepackage{cleveref}
\usepackage[all,cmtip]{xy}
\usepackage{tikz-cd}
\usepackage{mathtools}
\usepackage{amsmath,calligra,mathrsfs}

\usepackage{todonotes}

\numberwithin{equation}{subsection}

\newtheorem{theorem}[subsection]{Theorem}
\newtheorem{lemma}[subsection]{Lemma}
\newtheorem{proposition}[subsection]{Proposition}
\newtheorem{corollary}[subsection]{Corollary}
\theoremstyle{definition}

\newtheorem{defn}[subsection]{Definition}

\newtheorem{conjecture}[subsection]{Conjecture}

\newtheorem*{conjecture*}{Conjecture}
\newtheorem*{theorem*}{Theorem}
\newtheorem*{claim*}{Claim}
\newtheorem*{corollary*}{Corollary}

\newcommand{\pdiv}{\mathscr{G}}
\newcommand{\fH}{\mathscr{H}}

\newcommand{\hdiv}{\mathscr{H}}

\newcommand{\D}{\mathbb{D}}

\newcommand{\F}{\mathbb{F}}
\newcommand{\Fpbar}{\overline{\mathbb{F}}_p}
\newcommand{\G}{\mathbb{G}}
\newcommand{\bH}{\mathbb{H}}

\newcommand{\bL}{\mathbb{L}}

\newcommand{\gn}{\mathfrak n}

\newcommand{\Q}{\mathbb{Q}}
\newcommand{\breveQ}{\breve{\Q}}

\newcommand{\Z}{\mathbb{Z}}

\newcommand{\cA}{\mathcal{A}}

\newcommand{\cO}{\mathcal{O}}

\newcommand{\cS}{\mathcal{S}}

\newcommand{\fm}{\mathfrak{m}}

\newcommand{\X}{\mathbb{X}^{\textrm{u}}}

\newcommand{\lps}{[\![}
\newcommand{\rps}{]\!]}

\newcommand{\et}{{\text{\'et}}}

\newcommand{\gS}{\mathfrak S}
\newcommand{\gM}{\mathfrak M}

\newcommand{\bilinear}{\langle \cdot,\cdot \rangle}

\newcommand{\cris}{{\mathrm{cris}}}

\newcommand{\Fil}{{\mathrm{Fil}}}

\DeclareMathOperator{\GL}{GL}

\DeclareMathOperator{\SO}{SO}
\DeclareMathOperator{\GSO}{GO}

\DeclareMathOperator{\GSpin}{GSpin}

\DeclareMathOperator{\opp}{opp}

\DeclareMathOperator{\Gal}{Gal}
\DeclareMathOperator{\End}{End}

\newcommand*{\sheafhom}{\mathscr{H}\kern -1.5pt om}
\DeclareMathOperator{\Aut}{Aut}
\DeclareMathOperator{\Lie}{Lie}
\DeclareMathOperator{\Res}{Res}
\DeclareMathOperator{\Spec}{Spec}

\DeclareMathOperator{\der}{der}

\DeclareMathOperator{\Isoc}{Isoc}
\DeclareMathOperator{\Rep}{Rep}

\DeclareMathOperator{\height}{height}
\DeclareMathOperator{\Spf}{Spf}
\DeclareMathOperator{\Frob}{Frob}
\renewcommand{\ker}{{\rm Ker}}

\providecommand{\keywords}[1]
{
  \small	
  \textbf{\textit{Keywords---}} #1
}

\begin{document}
\title{Finiteness of reductions of Hecke orbits}

\author{Mark Kisin}
\address{Current: M.~Kisin, Department of Mathematics, Harvard University,
1 Oxford St, Cambridge, MA 02138, USA}
\email{kisin@math.harvard.edu}

\author{Yeuk Hay Joshua Lam}
\address{Current: Y.H.J.~Lam, Department of Mathematics, Harvard University,
1 Oxford St, Cambridge, MA 02138, USA}
\email{ylam@g.harvard.edu}

\author{Ananth N.Shankar}
\address{Current: A.N.~Shankar, Department of Mathematics, University of Wisconsin, Madison, 480 Lincoln Drive
213 Van Vleck Hall, Madison, WI 53706.}
\email{ashankar@math.wisc.edu}

\author{Padmavathi Srinivasan}
\address{Current: P.~Srinivasan, School of Mathematics, University of Georgia, 452 Boyd Graduate Studies, 1023 D. W. Brooks Drive, Athens, GA 30602.}
\email{Padmavathi.Srinivasan@uga.edu}


\date{\today}

\begin{abstract}
We prove two finiteness results for reductions of Hecke orbits of abelian varieties over local fields: one in the case of supersingular reduction and one in the case of reductive monodromy. As an application, we show that only finitely many abelian varieties on a fixed isogeny leaf admit CM lifts, which in particular implies that  in each  fixed dimension $g$ only finitely many supersingular abelian varieties admit CM lifts. Combining this with the Kuga-Satake construction, we also show that only finitely many supersingular $K3$-surfaces admit CM lifts. Our tools include $p$-adic Hodge theory and  group theoretic techniques. 
\end{abstract}


\maketitle

\section{Introduction}
Let $\bar{A}$ be an abelian variety over $\Fpbar.$ When $\bar{A}$ is \emph{ordinary}, then $\bar{A}$ admits a canonical (CM) lift, and every isogeny from $\bar{A}$ lifts to an isogeny in characteristic zero with source any fixed lift of $\bar{A}$. The aim of this paper is to show that the situation is radically different for {\em supersingular} 
abelian varieties. In fact, we prove a general theorem for all Newton strata that interpolates between the ordinary case and the supersingluar case.

The first example of an abelian variety over $\Fpbar$ without a CM lift was given by Oort in \cite{oortjag}. Further examples of such abelian varieties, including supersingular abelian varieties, were then constructed by Conrad, Chai and Oort in \cite{conradcmlifting}. We prove that such examples are in fact quite abundant.

\begin{theorem}\label{thm:introssthm}
Only finitely many supersingular abelian varieties of a given dimension admit CM lifts.
\end{theorem}

Combining this with a refined analysis of the Kuga-Satake construction, we are able to answer a question of Ito-Ito-Koshikawa \cite[Remark 1.3]{itocmlift}
\begin{theorem}\label{C:introK3thm} If $p \geq 5,$ then only finitely many supersingular $K3$-surfaces over $\Fpbar$ admit CM lifts.
\end{theorem}

We remark that the theorem above makes no mention of polarizations, so does not simply follow from the Kuga-Satake  map for integral models of Shimura varieties. It instead requires an analysis of a Kuga-Satake construction at the level of $p$-divisible groups due to Yang \cite{ziquansymplectic}.

Our results for general Newton strata use the notion of central leaves introduced by Oort in \cite[Theorem~5.3]{oortfoliation}.
The central leaf through $\bar{A}$ is a closed subvariety inside its Newton stratum, and essentially consists of all abelian varieties whose $p$-divisible group is geometrically isomorphic to that of $\bar{A}$. In \cite[Section 5]{oortdimension}, Oort computes the dimensions of central leaves and shows that they are $0$-dimensional in the supersingular stratum, and equal the entire Newton stratum in the ordinary case. 

\begin{theorem}\label{thm:introleavesthm}
Let $W$ be a Newton stratum in the moduli space of principally polarized abelian varieties in characteristic $p$. The set of points of $W$ admitting CM lifts is contained in a finite union of central leaves.
\end{theorem}

Let $A$ be an abelian variety over a field $K,$ with algebraic closure $\bar K.$  
By the \emph{Hecke orbit} of an abelian variety $A$, we mean the set of isomorphism classes of abelian varieties over $\bar K,$ which are isogenous to $A_{\bar K}.$  
If $A$ is defined over a local field, and has good \emph{reduction} $\bar A,$ then the image of the Hecke orbit of $A$ in the Hecke orbit of $\bar A$ is called the {\em reduction} of the former. We prove the following result, which is in stark contrast to the ordinary case. 
This is the key input for proving the CM lifting theorems stated earlier.
\begin{theorem}\label{thm:introisothm}
Let $A$ be an abelian variety over a characteristic zero local field, and suppose $A$ has good supersingular reduction. 
Then the  reduction of its Hecke orbit is finite.
\end{theorem}

This theorem answers a question posed by Poonen in an unpublished preprint, and also makes progress towards understanding the $p$-adic distribution of Hecke orbits.  The proof of Theorem~\ref{thm:introisothm} is entirely local, and we prove an analogous theorem in the setting of $p$-divisible groups first. 

Now suppose that $K$ is a characteristic zero local field with ring of integers $\cO_K$ and residue field $k.$ 
For a $p$-divisible group $\pdiv$ over $\cO_K,$ 
we define the Hecke orbit of $\pdiv$ and its reduction in the same way as above. In particular, the reduction is a collection of isomorphism 
classes of $p$-divisible groups over an algebraic closure of $k.$ 
We also establish a finiteness theorem under a semisimplicity hypothesis on the $p$-adic Galois representation.
\begin{theorem}[Theorem~\ref{newtred}]\label{thm:intropdivthm}
Let $\pdiv$ denote a $p$-divisible group over $\cO_K,$ such that the $p$-adic Galois representation associated to $\pdiv$ is semisimple. Then the reduction of the Hecke orbit of $\pdiv$ is finite.
\end{theorem}
 
In light of Theorem~\ref{thm:introisothm} and Theorem~\ref{thm:intropdivthm}, we make the following conjecture.
\begin{conjecture}\label{conj:pdivorbit}
For any $p$-divisible group $\pdiv$ over $\cO_K$, the reduction of its Hecke orbit is finite. 
\end{conjecture}

By Theorem \ref{thm:introssthm} and Theorem \ref{thm:intropdivthm}, we know that Conjecture \ref{conj:pdivorbit} holds in the case of supersingular reduction (without any semisimplicity conditions on the Galois representation), and in the case that the Galois representation is semisimple (without any condition on the Newton polygon of $\pdiv$). We remark that unlike the situation in characteristic zero the Hecke orbit of $\bar{A}$ can contain 
positive dimensional families\footnote{This always happens unless the $p$-rank of the $g$-dimensional abelian variety is equal to $g$ or $g-1$.}. This was first observed by Moret-Bailly \cite{moretbailly}, who constructed a complete family of supersingular abelian surfaces over $\mathbb{P}^1_{\mathbb{F}_p}$, such that all fibers are $p$-isogenous. 

\textbf{Outline of the paper:}
 In \S~\ref{FinitenessThms}, we prove our results on finiteness of reductions of Hecke orbits. This is done using a Galois-theoretic 
 result which relies on work of Sen and Serre. We apply this in \S~\ref{section:cmlift} to show the results on CM lifts of $p$-divisible groups 
 and abelian varieties. Here we make crucial use of Oort's results on central leaves. Finally in \S~\ref{sec:CMK3}, we prove the finiteness result for CM lifts of supersingular $K3$ surfaces, 
 by comparing the deformation theory of $K3$'s with that of $\GSpin$ $p$-divisible groups. 
 
 \section*{Acknowledgements}
 We are very grateful to Bjorn Poonen for sharing his unpublished pre-print where we first saw the question about reductions of Hecke orbits, and for useful conversations. We are extremely grateful to Arul Shankar for several valuable conversations pertaining to CM lifts of abelian varieties. We also thank Ching-Li Chai and Ziquan Yang for valuable discussions and correspondence pertaining to CM lifts and K3 surfaces. 
  
 \section{Finiteness for reductions of Hecke orbits}\label{FinitenessThms} 

\subsection{} Let $K$ be a field equipped with a rank $1$ valuation. Throughout the paper, we will denote by $\cO_K$ the ring 
of integers of $K.$ Fix an algebraic closure $\bar K$ of $K.$  We denote by $G_K = \Gal(\bar K/K)$ the absolute Galois group of $K,$ and by $I_K \subset G_K$ the inertia subgroup. 

In this section we suppose that $K$ is a finite extension of $\Q_p.$ We write  $\bar \F_p$ for the residue field of $\bar K.$ We denote by $\breveQ_p$ the maximal unramified extension of $\Q_p$, and by $\breve{K}$ the compositum of $K$ and $\breveQ_p$.

\subsection{} 
 Let $\pdiv/\cO_K$ denote a $p$-divisible group. 
Denote by $T_p\pdiv$ the $p$-adic Tate module of $\pdiv,$ and let 
$\rho = \rho_{\pdiv}: G_K \rightarrow \GL(T_p\pdiv)$ be the Galois representation 
associated to $\pdiv.$ We denote by $G$ (resp. $H$) the Zariski closure of $\rho(G_K)$ (resp. $\rho(I_K)$) in $\GL(T_p\pdiv[1/p]).$ 
Note that, since $\rho(I_K)$ is normal in $\rho(G_K),$ $H$ is normal in $G.$

\begin{lemma}\label{zariskiclosure}
We have  $[G:H]< \infty$ if and only if $[\rho(G_K):\rho(I_K)] < \infty$.
\end{lemma}
\begin{proof}
If $\rho(G_K) = \displaystyle{\cup_{g \in S}} g \rho(I_K)$ for some finite set $S$, then $ \cup_{g \in S} g H$ 
is a closed set containing $\rho(G_K)$ and hence equals $G.$ This shows that $[G:H]$ is finite if $[\rho(G_K):\rho(I_K)]$ is. 

For the other direction, suppose that $[G:H]<\infty$. By a Theorem  of Sen (\cite[Theorem 2]{sen}) and Serre (\cite[Theorem 1]{serre1}), $\rho(I_K)$ is open in $H(\Q_p),$ 
and hence in $G(\Q_p),$ as $H(\Q_p)$ has finite index in $G(\Q_p).$ In particular, this implies that $\rho(I_K)$ is open 
in $\rho(G_K) \subset G(\Q_p).$ Since $\rho(G_K)$ is compact, this implies $\rho(G_K)/\rho(I_K)$ is discrete and compact, hence finite.
\end{proof}

\begin{defn}
We say that two $p$-divisible groups over a finite field $\mathbb{F}_q$ are equivalent if they become isomorphic over $\Fpbar$.
\end{defn}

\begin{lemma}\label{oortvasiu}
 For any $h\geq 1,$ the set of equivalence classes of $p$-divisible groups over $\F_q$ of height $h$ is finite. 
 \end{lemma}

\begin{proof} By a result of Oort \cite[Corollary 1.7]{oortfoliation}, there is an integer $n=n(h)$ such that any two 
$p$-divisible groups over $\Fpbar$ of height $h$, are isomorphic if and only if their $p^n$-torsion subgroups are isomorphic. 
In particular, the equivalence class of a $p$-divisible group $\fH$ over $\F_q$ of height $h$ is determined by its $p^n$-torsion subgroup 
$\fH[p^n].$ Since $\fH[p^n]$ is a finite flat group scheme over $\F_q$ of order $p^{nh},$ there are only finitely many possibilities for 
$\fH[p^n],$ and the lemma follows.
\end{proof}

\subsection{}\label{defn:redheckeorbit} Let $J(\pdiv)$ denote the set of isomorphism classes of $p$-divisible groups over $\cO_{\bar K}$ which are isogenous to $\pdiv\otimes_{\cO_K}\cO_{\bar K}.$ We set 
\[
J(\pdiv, \Fpbar) \colonequals \{\pdiv'\otimes_{\cO_{K'}}\bar \F_p \ | \  \pdiv' \in J(\pdiv)\},
\]
the set of isomorphism classes of reductions of elements of $J(\pdiv).$

Analogously, we denote by $I(A)$ the Hecke orbit of $A$, namely the set isomorphisms of abelian varieties over $\cO_{\bar K}$ which are isogenous to $A\otimes_{\cO_K}\cO_{\bar K}.$ We then define   
\[
I(A, \Fpbar) \colonequals \{A' \otimes_{\cO_{K'}}\bar \F_p\ | \  A' \in I(A)\},
\]
 the reduction of the Hecke orbit of $A$.

We now give a Galois-theoretic criterion for the finiteness of $J(\pdiv,\Fpbar)$
\begin{proposition}\label{propidea} If $[G:H]< \infty,$ then $J(\pdiv, \Fpbar)$ is a finite set. 
\end{proposition}
\begin{proof} By Lemma \ref{zariskiclosure}, our hypothesis implies that $\rho(I_K)$ has finite index in $\rho(G_K)$. Therefore, after replacing $K$ by a finite extension if necessary, we may assume that $\rho(I_K)=\rho(G_K).$ 
Let $K_{\rho}$ the splitting field of $\rho$, namely the Galois extension defined by the subgroup $\ker(\rho)$.
Since $\rho(I_K)=\rho(G_K),$ $K_{\rho}/K$ is a totally ramified extension.

Note that any isogeny of $p$-divisible groups with source $\pdiv$ can be defined over $K_{\rho}.$ 
Hence if $\pdiv'$ is in $J(\pdiv),$ then $\pdiv'$ has a representative (again denoted $\pdiv'$) which is defined over $K_{\rho}.$ 
Since $K_{\rho}/K$ is totally ramified, the reduction $\pdiv'\otimes_{\cO_{\bar K}} \bar \F_p$ is defined over the residue field of $K,$  namely 
$\F_q.$ The finiteness of $J(\pdiv,\bar \F_p)$ now follows from Proposition \ref{oortvasiu}.\qedhere
\end{proof}

\subsection{} We will apply Proposition \ref{propidea} in two cases. To explain the first of these, recall that 
for a $p$-divisible group $\fH$ over $\F_q,$ its Dieudonn\'e module $\D(\fH)$ is a finite free $W(\F_q)$-module equipped 
with a semi-linear Frobenius $\varphi.$ If $q = p^r,$ then $\varphi^r$ acts linearly on $\D(\fH),$ and we call this action the 
$q$-Frobenius $\Frob_q$ on $\D(\fH).$ Following Rappoport-Zink \cite{rapoportzink}, we say that $\hdiv$ is \emph{decent} if the action of $\Frob_q$ on $\D(\fH)$ is 
semisimple with eigenvalues which are all rational powers of $q.$ That latter condition means that for each eigenvalue 
$\alpha,$ $\alpha^m = q^n$ for some integers $m,n.$

The semisimplicity condition is always satisfied if $\fH = B[p^\infty]$ is the $p$-divisible group arising from an abelian variety $B$ over $\F_q.$ 
Examples of a decent $p$-divisible groups include those of the form $\fH = B[p^\infty]$ with $B$ a supersingular abelian variety 
$B$ over $\F_q.$ 

\begin{proposition}\label{prop:usingdecent} If  $\fH = \pdiv\otimes_{\cO_K} \F_q$ is decent then $[G:H]$ is finite. 

\end{proposition}
\begin{proof} Let $\Rep_G$ denote the $\Q_p$-linear Tannakian category of algebraic representations of $G,$ 
and $\Isoc_{\Q_p}$ the Tannakian category  of isocrystals over $K_0,$ the maximal absolutely unramified subfield of $K.$ 
As for $\D(\fH),$ $\Frob_q: = \varphi^r$ acting on $\Isoc_{\Q_p}$ is linear. 
Using Fontaine's functor $D_{\cris},$ the representation $\rho:G_K \rightarrow G(\Q_p),$ gives rise to a functor  
\[
D_{\cris}: \Rep_G \rightarrow \Isoc_{\Q_p} \quad w \mapsto D_{\cris}(w\circ \rho).
\]
This functor sends $V = (T_p\pdiv[1/p])^*$ to $\D[1/p],$ where $\D = \D(\fH).$

Now let $W$ be a faithful representation of $G/H,$ viewed as a representation of $G,$ and write $w: G \rightarrow \GL(W).$ 
Then $\rho\circ w$ is an unramified representation, so that the eigenvalues of $\Frob_q$ acting on $D_{\cris}(W)$ are $p$-adic 
units. Since $W$ is in the Tannakian category generated by $V$, it follows that $D_{\cris}(W)$ is in the Tannakian category generated by 
$\D[1/p].$ Hence $\Frob_q$ acting on $D_{\cris}(W)$ is semi-simple and each of its eigenvalues have the 
form $\prod_{i=1}^n \alpha_i^{d_i},$ where $\alpha_1, \cdots , \alpha_n$ are the eigenvalues of $\Frob_q$ on $D_{\cris}(V).$ 
In particular any eigenvalue $\alpha$ of $\Frob_q$ acting on $D_{\cris}(W)$ is a rational power of $q.$ 
Since $\alpha$ is also a $p$-adic unit, it follows that $\alpha$ is a root of unity. 

It follows that $\Frob_q$ acting on $D_{\cris}(W)$ has finite order. This implies that for some power $q'$ of $q,$ $W,$ viewed as a representation of $G_K/I_K,$ can be identified with $(D_{\cris}(W)\otimes W(\F_{q'}))^{\varphi=1},$ with $G_K/I_K$ acting 
via its action on $W(\F_{q'}).$ In particular, we see that $G_K/I_K$ acts on $W$ through a finite quotient. Since $\rho(G_K)$ 
is dense in $G,$ it is dense in the image of $G$ in $\GL(W),$ and hence this image is finite. As $W$ was a faithful representation 
of $G/H,$ this proves the proposition.
\end{proof}

\begin{corollary}\label{unramtofinite} If $\fH = \pdiv\otimes_{\cO_K} \F_q$ is decent, then $J(\pdiv,\Fpbar)$ is a finite set. 
If moroever $\pdiv = A[p^\infty]$ for an abelian scheme $A$ over $\cO_K$ with supersingular reduction, then $I(\pdiv, \Fpbar)$ is a finite set. 
\end{corollary}
\begin{proof} The first statement follows immediately Proposition \ref{prop:usingdecent} and Proposition \ref{propidea}.

For the second statement, as the special fiber of $A$ is supersingular, we may replace $K$ by a finite extension so that the Galois action on the prime-to-$p$ torsion of $A$ is through scalars (with Frobenius mapping to the scalar $q^{1/2}$ where $q$ is the size of the residue field of $K$). Further, we observe that $A[p^{\infty}]\otimes_{\cO_K} \F_q$ is decent as $A$ has supersingular reduction, and hence the hypothesis of Proposition \ref{propidea} holds. Let $K_{\rho}$ be as in the proof of Proposition \ref{propidea}. We have that every isogeny from $A$ is defined over $K_{\rho}$ and hence all $A'$ isogenous to $A$ have reductions defined over $\F_q$, the residue field of $K$.  By Zarhin's trick \cite[Theorem 4.1]{zarhin} there are only finitely many isomorphism classes of abelian varieties over $\mathbb{F}_q$, and hence $I(A,\Fpbar)$ is a finite set.
\end{proof}

\subsection{} Our second application of Proposition \ref{propidea} is more indirect, and proceeds by showing that even if 
$\pdiv$ does not satisfy the hypothesis of Proposition \ref{propidea} one can sometimes construct an auxiliary $p$-divisible group which 
does.

\begin{lemma}\label{lemma:centraltwist} Suppose that the connected component of the identity in $G$ is reductive. 
Then after replacing $K$ by a finite extension,  there exists a $p$-divisible group $\pdiv'/\cO_K$ 
such that 
\begin{enumerate}
\item $\pdiv$ and $\pdiv'$ are isomorphic over $\cO_{\breve {K}}$
\item If $G'$ (resp.~$H'$) denotes the Zariski closure of the image of $G_K$ (resp.~$I_K$) in $\GL(T_p\pdiv'[1/p]),$ 
then $G' = H'.$
\end{enumerate}
\end{lemma}
\begin{proof} After replacing $K$ by a finite extension, we may assume that $G$ is reductive, and we set $T = G/H.$ 
As $T$ is abelian, it is a torus. Let $Z_G$ denote the center of $G,$ and $G^{\der}$ its derived subgroup. 
The map $Z_G \times G^{\der} \rightarrow G$ is surjective with finite kernel, so we obtain a surjective map $Z_G \rightarrow T.$ 
This implies that there is a subtorus $T_{\mathrm{sub}} \subset Z_G,$ such that the map $T_{\mathrm{sub}} \rightarrow T$ is an isogeny.

Let $\chi: G_K \rightarrow T(\Q_p)$ be the map induced by $\rho,$ and let $\sigma \in G_K$ be a lift of the $q$-Frobenius. 
For some positive integer $m,$ $\tau = \chi(\sigma^m)$ lifts to an element of $T_{\mathrm{sub}}(\Q_p).$ Thus after replacing $K$ by a finite extension, we may assume that $\tau$ lifts to an element $\tau' \in T_{\mathrm{sub}}(\Q_p).$ Since $T_{\mathrm{sub}} \rightarrow T$ is an isogeny, 
the subgroup generated by $\tau'$ is bounded, so there is an  unramified Galois character of the form
$$
\psi: G_K\rightarrow G_K/I_K \rightarrow T_{\mathrm{sub}}(\Q_p); \quad \sigma \mapsto \tau'
$$
Now let $C \subset Z_G(\Q_p)$ denote the subgroup preserving $T_p\pdiv.$ 
After replacing $K$ by a finite extension, and so $\tau'$ by a power, we may assume that $\tau' \in C.$ 

The action of $C$ on $T_p\pdiv$ commutes with the action of $G_K.$ Hence by Tate's theorem, it induces 
a map $C \rightarrow \Aut \pdiv.$ Since $\pdiv$ is defined over $\cO_K,$ for any $\gamma \in G_K/I_K,$ we have a canonical isomorphism
$\gamma^* \pdiv \simeq \pdiv$ over $\cO_{\breve{K}}.$ Denote by $c_{\gamma}$ the composite of this isomorphism 
and the automorphism $\psi(\gamma)^{-1} \in C$ viewed as an automorphism of $\pdiv$ (here we view $\psi$ as a character on 
$G_K/I_K$). Then $c_{\gamma}$ defines a descent datum on $\pdiv[p^n]|_{\cO_{\breve{K}}},$ for each $n.$ By \'etale descent, 
$c_{\gamma}$ arises from a unique $p$-divisible group $\pdiv'$ over $\cO_K.$ 

By construction, the underlying $\Z_p$-modules of $T_p\pdiv$ and $T_p\pdiv'$ are canonically identified, and the action of $G_K$ on $\pdiv'$ is obtained by multiplying its action on $T_p\pdiv$ by $\psi(\sigma)^{-1}.$ Hence we have 
$ H = H' \subset G' \subset G,$ and the map $G' \rightarrow T$ is trivial. It follows that $G' = H'.$
\end{proof}

\begin{corollary}\label{newtred} If the $p$-adic Galois representation associated to $\pdiv$ is semisimple, then the set $J(\pdiv,\Fpbar)$ is finite. 
\end{corollary}
\begin{proof} By Lemma \ref{lemma:centraltwist} we can find $\pdiv'$ such that $H'=G'$. By Proposition \ref{propidea} $J(\pdiv', \Fpbar)$ is a finite set.  On the other hand the two $p$-divisible groups $\pdiv$ and $\pdiv'$ are isomorphic over $\breve K$ and therefore 
$J(\pdiv, \Fpbar)=J(\pdiv', \Fpbar)$.
\end{proof}

\section{Finiteness of $p$-divisible groups admitting a CM lift}\label{section:cmlift}

\subsection{} \label{S:padicCMtype} 
In this section we assume that $K$ is a finite extension of $K_0 = W(\bar \F_p)[1/p].$ 

A $p$-divisible group $\pdiv$ of (constant) height $h,$ over any base, is said to have CM by a commutative semisimple $\mathbb{Q}_p$-algebra $F$ if there is an injective homomorphism
\[
F\hookrightarrow\End(\pdiv)\otimes \mathbb{Q},
\]
such that $\dim_{\Q_p} F=\height(\pdiv)$.
We say that $\pdiv$ is CM, or has CM if $\pdiv$ has CM by some $F$ as above.

If $\pdiv$ is a $p$-divisible group over $\cO_K,$ we can form its formal group $\widehat \pdiv.$ 
If $\pdiv$ has CM by $F,$ then $\Lie \widehat \pdiv \otimes_{\cO_K} \bar K \simeq \oplus_{\sigma} V_{\sigma}$ 
where $\sigma$ runs over $\Q_p$-algebra maps $F \rightarrow \bar K,$ and for $a \in F,$ we have $aV_{\sigma} = \sigma(a)V_{\sigma}.$
 For each $\sigma,$ the summand $V_{\sigma}$ is either trivial, or one dimensional over $\bar K$ \cite[Lemma 3.7.1.3]{conradcmlifting}. 
We denote by $\Phi$ the set of $\sigma$ for which $V_{\sigma}$ is one dimensional, and we call $\Phi$ the CM type of $\pdiv.$

\begin{lemma}\label{lem:descenttolocal} Let $\pdiv$ be a $p$-divisible group over $\cO_K$ with CM by $F.$ Then there exists 
a finite extension $K'/\Q_p$ contained in $K,$ and a $p$-divisible group $\pdiv'$ over $\cO_{K'}$ with CM by $F$ such that  
$\pdiv'\otimes_{\cO_{K'}}\cO_K$ is $F$-linearly isomorphic to $\pdiv.$
\end{lemma}
\begin{proof} This is well known. Let $D$ be the weakly admissible module $D_{\cris}(T_p\pdiv[1/p]),$ so that $D$ is an  
$F\otimes_{\Q_p}K_0$-module equipped with an injective, semi-linear Frobenius and a one step filtration 
$\Fil^1 D_K \subset D_K = D\otimes_{K_0}K$ by an 
$F\otimes_{\Q_p}K_0$-submodule. 

The filtration on $D_K$ is induced by a cocharacter $\mu \in X_*(\Res_{F/\Q_p} \G_m).$ 
Choose $K'/\Q_p$ finite and contained in $K,$ such that $\mu$ is defined over $K'$, and let $K'_0$ denote the maximal unramified subfield of $K'$. Then there exists a free $F\otimes_{\Q_p}K'_0$-module $D',$ an $F\otimes_{\Q_p}K'$-submodule $\Fil^1 D'_{K'} \subset D'_{K'} = D'\otimes K',$ and an $F$-linear isomorphism 
$\iota:D'\otimes_{K_0'}K_0 \simeq D,$ respecting filtrations. 

Now identify $D'$ with $F\otimes_{\Q_p}K'_0.$ Then we can identify $D$ with  $F\otimes_{\Q_p}K_0$ via $\iota,$ and the Frobenius on $D$ given by $\delta\sigma,$ where $\delta \in (F\otimes_{\Q_p}K_0)^\times$ and $\sigma$ denotes the Frobenius on $K_0.$ After possibly replacing $K'$ by a larger field, there exists $\delta' \in (F\otimes_{\Q_p}K'_0)^\times$ such that 
$\delta'\delta^{-1}$ is unit. We equip $D'$ with the Frobenius $\delta'\sigma.$ 

As $\delta'\delta^{-1}$ is a unit there exists $c \in (F\otimes_{\Q_p}K_0)^\times$ such that $\delta' = c^{-1}\delta\sigma(c).$ 
Then $c\cdot \iota$ respects Frobenius, and also respects filtrations as $\iota$ does. $D'$ along with the Frobenius $\delta'\sigma$ and filtration $\Fil^1 D'_{K'} \subset D'_{K'}$ is a weakly admissible module as $D$ is. This weakly admissible module equals $ D_{\cris}(T_p \pdiv'[1/p])$ where $\pdiv'/\cO_{K'}$ is a $p$-divisible group, for example by \cite{kisinfcrystal}. The isomorphism $D'_{K_0}\simeq D_{K_0}$ induces a quasi-isogeny between the Tate-modules of $\pdiv'$ and $\pdiv$ as $\Gal(\bar{K}/K)$-representations, and hence (after multiplication by a power of $p$) an isogeny $\pdiv'_{\cO_K} \rightarrow \pdiv$. After replacing $K'$ by a finite extension, we may assume that the kernel of this isogeny is defined over $\cO_{K'}$ and the theorem follows. 
\end{proof}

\subsection{} Let $\hdiv$ be a $p$-divisible group over $\bar \F_p.$ We say that $\hdiv$ admits a CM lift, if there 
exists a finite extension $K/W(\bar \F_p)[1/p],$ and a CM $p$-divisible group $\pdiv$ over $\cO_K,$ such that 
$\pdiv\otimes_{\cO_K} \bar \F_p$ is isomorphic to $\hdiv.$ 

We remark that there is essentially no extra generality gained by considering CM lifts to more general base rings. 
More precisely, if $R$ is an integral, normal, flat $W(\bar \F_p)$-algebra, and $\pdiv$ is a CM deformation of $\hdiv$ to 
$R,$ then there is a finite extension $K/W(\bar \F_p)[1/p],$ and an inclusion $\cO_K \rightarrow R,$ such that $\pdiv$ 
arises from a CM deformation of $\hdiv$ over $\cO_K.$ This can be deduced from the fact that the rigid analytic period 
morphism in \cite[\S 5]{rapoportzink} is \'etale, together with the fact that any $F\otimes_{W(\bar \F_p)} R$-direct summand 
of the free, rank $1$ $F\otimes_{W(\bar \F_p)} R$-module $\D(\pdiv)(R)$ is defined over $F\otimes_{W(\bar \F_p)} \cO_K,$ 
for some $\cO_K \subset R,$ as above.

\begin{theorem}\label{thm:cmliftpdiv}
  Let $\hdiv/\bar \F_p$ be a $p$-divisible group. Then, the isogeny class of $\hdiv$ contains only finitely many isomorphism 
  classes of $p$-divisible groups which admit a CM lift. 
\end{theorem}
\begin{proof}
Since the algebra $\End(\hdiv)\otimes \Q$ has finite dimension over $\Q_p$, there are only finitely many choices for the CM algebra $F.$ 
Given $F,$ there are only finitely many choices for the CM type $\Phi.$ Thus, we may fix $F$ and $\Phi,$ and consider only 
$p$-divisible groups in the isogeny class of $\hdiv,$ which admit a CM lift having CM by $F$ and CM type $\Phi.$

Let $\pdiv, \pdiv_1$ be such lifts, defined over some finite extension $K/W(\bar \F_p)[1/p].$ By Lemma \ref{lem:descenttolocal}, there exists a $p$-divisible group $\pdiv'$ with CM by $F,$ defined over a finite extension $K'/\Q_p$ such that 
$\pdiv'\otimes_{\cO_{K'}}\cO_{K} \simeq \pdiv.$ Since $\pdiv_1$ and $\pdiv$ have the same CM type, there is an $F$-linear 
isogeny $\pdiv \rightarrow \pdiv_1,$ by \cite[Proposition 3.7.4]{conradcmlifting}. Thus $\pdiv_1 \in J(\pdiv').$ 

Now the Zariski closure of the image of $G_{K'}$ acting on $T_p\pdiv'$ is a closed subgroup of the torus $\Res_{F/\Q_p} \G_m,$ 
hence is reductive. Hence $J(\pdiv', \bar \F_p)$ is finite by Theorem \ref{newtred}, and the theorem follows.
\end{proof}

\subsection{} Let $\cA_g$ denote the moduli space of principally polarized abelian varieties of dimension $g$. 
For a given Newton polygon $\nu$ we denote by $W_{\nu} \subset \cA_{g,\bar \F_p}$ the corresponding Newton stratum. 
It is a locally closed subscheme. 

For any $x \in \cA_g(\bar \F_p),$ the associated $p$-divisible group $\pdiv_x$ carries a principal polarization, 
an isomorphism, $\psi_x,$ of $\pdiv_x$ and its Cartier dual. The polarized central leaf through a point $x\in \mathcal{A}_g$ is the locus of points where the associated polarized $p$-divisible group is geometrically \emph{isomorphic} to the polarized $p$-divisible group parameterized by $x.$  Oort \cite[Theorem~5.3]{oortfoliation} has shown that this locus is a closed subvariety of the Newton stratum through $x.$

\begin{theorem}\label{thm:cmlift}
Let $S$ be the set of points in $\mathcal{W}_{\nu}$ which admit CM lifts. Then $S$ is contained in a finite number of central leaves. 
\end{theorem}
\begin{proof} 
Let $n$ denote some large enough integer, such that $\mathcal{A}_g[n]$, the moduli space of principally polarized $g$-dimensional abelian varieties with full symplectic level $n$ structure, is a fine moduli space. The notion of a polarized central leaf in $\mathcal{A}_g[n]$ through any point is defined exactly as in the case of $\mathcal{A}_g$, and without any reference to level structure. Then, the result for $\mathcal{A}_g$ follows directly from the result for $\mathcal{A}_g[n]$.

Let $x \in \mathcal{W}_{\nu}(\Fpbar)\subset \cA_g[n](\bar \F_p),$ and let $\hdiv_x$ and $\cA_x$ be the associated principally polarized $p$-divisible group and  
abelian scheme, respectively. By \cite[Theorem 2.2]{oortfoliation}, the set of points $y\in \mathcal{A}_g[n]$ with $\hdiv_y$ isomorphic to $\hdiv_x$ (as unpolarized $p$-divisible groups) is a closed subvariety of $\mathcal{W}_{\nu}$. By \cite[Theorem 3.3]{oortfoliation}, this closed subvariety is a union of finitely many central leaves\footnote{Both of Oort's results are stated in the setting of families of $p$-divisible groups, which is the only reason we need to work with $\cA_g[n]$ instead of $\cA_g$.} in $\mathcal{W}_{\nu}$. The result now follows immediately from Theorem \ref{thm:cmliftpdiv}.
\end{proof}

\begin{corollary}\label{cor:sscmlift} There are only finitely many supersingular principally polarized abelian varieties of dimension $g,$ 
which admit a CM lift.
\end{corollary}
\begin{proof} This follows from Theorem \ref{thm:cmlift}, and the fact that central leaves in the supersingular stratum are zero-dimensional  (\cite[Section~5]{oortdimension}).
\end{proof}

\section{CM lifts of supersingular $K3$-surfaces}\label{sec:CMK3}
In this section we deduce Theorem~\ref{C:introK3thm} for CM lifts of $K3$-surfaces when $p\geq 5$. The weaker statement  for polarized $K3$-surfaces is an immediate consequence of Theorem~\ref{thm:cmlift} and the theory of integral models of Shimura varieties \cite{keerthi}. Here we prove the stronger result for $K3$-surfaces with no reference to  polarizations.  The main inputs will be the Kuga-Satake construction for $K3$-crystals (without any reference to polarizations) due to Yang, and the crystalline Torelli theorem, due to Ogus.

\subsection{}\label{subsection:backgroundK3} 
Let $R$ be a $p$-adically complete and separated, $p$-torsion free ring. 
Suppose that $R$ is equipped with a lift of Frobenius $\sigma.$ 
For $i \geq 0,$ we denote by $R(i),$ the $R$-module $R,$ equipped with a Frobenius 
$\varphi = p^i\cdot\sigma,$ and a filtration given by where $\Fil^2 R(2) = R(2),$ and  $\Fil^3 R(2) = 0.$

A $K3$-crystal over $R$ ([cf. \cite[Definition 3.1]{ogusk3}) is a free $R$-module $\bL$ of rank 22, endowed with a Frobenius-linear endomormophism $\varphi: \bL \rightarrow \bL$ and a $\varphi$-compatible,  
perfect, symmetric bilinear  form $\bilinear: \bL \times \bL \rightarrow R(2)$ satisfying: 
\begin{enumerate}
    \item[(a)] $p^2 \bL \subset \varphi(\bL)$. 
    \item[(b)] The image of $\varphi \otimes_R R/pR$ is projective of rank 1. 
        \end{enumerate}
When $R = W(k)$ for a perfect field $k,$ we also call this a $K3$-crystal over $k.$

A {\it filtered} $K3$-crystal over $R$ is a $K3$-crystal $\bL$ over $R,$ equipped with a 
filtration 
$0=\Fil^3 \subset \Fil^2 \subset \Fil^1 \subset \Fil^0 = \bL$ such that $\text{\rm gr}^{\bullet} \bL$ is a projective $R$-module 
and the following conditions hold 
\begin{enumerate}
    \item[(c)] $\bilinear: \bL\otimes\bL \rightarrow R(2)$ is strictly compatible with filtrations.
    \item [(d)] $\Fil^1 \bL \otimes _RR/pR$ is the kernel of $\varphi$ on $\bL\otimes_RR/R.$
    \end{enumerate}

Note that these conditions imply that $\Fil^1 = {\Fil^2}^{\perp},$ that $\varphi(\Fil^2 \bL) \subset p^2 \bL,$ and that for $i=0,1,2,$ $\text{\rm gr}^i \bL$ has rank $1,20,1$ respectively. 

\subsection{}\label{subsection:K3fields} Now suppose that $R = W(k)$ with $k$ a perfect field, which will be either $\bar \F_p$ or a finite field in application.

A $K3$-crystal $\bL$ over $k$ is said to be \emph{supersingular}, if the slopes of $\varphi$ are all 1.
If $\bL$ is supersingular, and $k$ is algebraically closed, then $\bL^{\varphi = p}$ is a free $\Z_p$-module of rank $22$ 
(\cite[Theorem 3.3]{ogusk3}), which also admits a bilinear form (which will no longer be perfect). 

If $k = \F_q$ is a finite field, then we say a $K3$-crystal $\bL$ over $k$ is {\em decent} if the $q$-Frobenius  on $\bL$ has  eigenvalues  which are rational powers of $q$. We note that this implies the $\Z_p$-module $\bL^{\varphi = p} \subset \bL$ has rank 22. 
Note that every $K3$-crystal over $\Fpbar$ admits a decent model over $\F_q$ for some $q,$ see \cite[\S 4.3]{kottwitz}

\begin{lemma}\label{parabolic} Let $\bL$ be a filtered $K3$-crystal over $R,$ as above. 
Then the filtration on $\bL$ is induced by a cocharacter $\mu: \mathbb G_m \rightarrow  \GSO(\bL, \bilinear).$ 
In particular, the subgroup $P \subset \GSO(\bL, \bilinear)$ preserving the filtration is parabolic.
\end{lemma}
\begin{proof} Let $\bL^2 = \Fil^2 \bL,$ and choose a submodule $\bL^1 \subset \Fil^1$ so that $\Fil^1 \bL = \bL^2 \oplus \bL^1.$ 
Then $(\bL^1)^\perp$ is free of rank $2$ and surjects on $\text{\rm gr}^0 \bL,$ as $\bilinear$ is perfect and strict for filtrations. 
Thus, we can choose a rank $1$ direct summand $\bL^0 \subset (\bL^1)^{\perp}$ which maps isomorphically to $\text{\rm gr}^0 \bL.$ 
Then $\bilinear$ induces a perfect pairing between the rank $1$ subspaces $\bL^0, \bL^2 \subset (\bL^1)^{\perp}.$ 
Thus, modifying $\bL^0 \subset (\bL^1)^{\perp}$ we may assume that $\bL^0$ is isotropic for $\bilinear.$

Now $\bL = \bL^2 \oplus \bL^1 \oplus \bL^0,$ and we define $\mu$ by requiring that $\mu(z)$ by $z^i$ on $\bL^i.$ 
Since $\bL^i$ and $\bL^j$ are orthogonal for $i+j \neq 2,$ $\mu(z)$ acts on $\bilinear$ by $z^2.$ 
In particular $\mu$ factors through $\GSO(\bL, \bilinear).$
\end{proof}

\subsection{}\label{subsec:K3setup} Now let $k$ be a perfect field of characteristic $p,$ which is either algebraically closed or an algebraic extension of $\F_p.$ Let $X/k$ denote a $K3$-surface. It is well known that the deformation functor of $X$ is smooth and pro-representable and formally smooth of dimension 20 over $W = W(k)$. Let $\Spf \Hat{R}_X$ denote the universal deformation space of $X$, and let $\X \xrightarrow{\pi} \Spf \Hat{R}$ denote the universal deformation of $X$. Choose a set $\{p,x_1 \hdots x_{20}\}$ of  elements that generate the maximal ideal of $\Hat{R}_X$; define $ \sigma$ to be the lift of the Frobenius endomorphism of $\Hat{R}_X \mod p$ such that $\sigma$ is the usual Frobenius on $W(\F_q)$, and $\sigma(x_i) = (x_i)^p.$

Then $\bL_u = R^2f_*\pi$ is a filtered $K3$-crystal over $\Hat{R}_X.$ To give a more explicit description of this filtered $K3$-crystal, we need the following 

\begin{lemma}\label{lem:descentZp} Let $\bL$ be the $K3$-crystal attached to $X.$ There exists a free $\Z_p$-module with quadratic form $(T,\bilinear')$ and an isomorphism $\iota: (T,\bilinear')\otimes_{\Z_p} W\rightarrow (\bL, \bilinear)$.
\end{lemma} 
\begin{proof} The bilinear form on $\bL$ is self dual. It follows from the theory of non-degenerate bilinear forms over finite fields that after replacing $k$ by at most a quadratic extension, $(\bL,\bilinear)$ has determinant a square and also admits a rank-11 isotropic subspace. Indeed, there is a unique quadratic form on $\bL$ (up to isomorphism) that satisfies these conditions, whence it follows that $(T,\bilinear')_{\Z_p}\otimes W$ is isomorphic to $(\bL,\bilinear)$ where $T$ is a rank-22 free $\Z_p$-module, and $\bilinear'$ is the unique self-dual quadratic form on $T$ that has square determinant and admits a rank-11 isotropic subspace. This weaker statement (namely, after replacing $k$ with a quadratic extension) will actually be enough for our applications, so we only sketch the proof of the full statement of the lemma.

Choose a lift $\tilde X$ of $X$ to a smooth formal scheme over $W.$ We will apply the theory of prismatic cohomology 
\cite{BhattScholze} to $\tilde X,$ using the map $\gS = W\lps u \rps \overset{u \mapsto 0}  \rightarrow W,$ where $\gS$ is equipped with 
the Frobenius which sends $u$ to $u^p;$ this is an example of a prism. We obtain a finite free $\gS$-module $\gM,$ 
equipped with a semi-linear Frobenius, and a bilinear pairing $\bilinear_{\gM}$ such that 
\begin{enumerate}
\item $(\gM, \bilinear_{\gM})\otimes_{\gS}W \simeq (\bL, \bilinear).$ 
\item There is a faithfully flat $\gS_{(p)}$ algebra $A$ and an isomorphism 
$$  (\gM, \bilinear_{\gM})\otimes_{\gS} A \simeq (H^2(\tilde X_{\bar K}, \Z_p), \bilinear)\otimes_{\Z_p}A $$
where $\bar K$ is an algebraic closure of $W[1/p],$ and $H^2(\tilde X_{\bar K}, \Z_p)$ is equipped 
with a perfect symmetric bilinear form, using Poincare duality.
\end{enumerate}

Then by the  Key Lemma of \cite[Prop.~1.3.4]{kisinintegral}, one finds that there exists an isomorphism 
$$ (\bL, \bilinear) \simeq (H^2(\tilde X_{\bar K}, \Z_p), \bilinear)\otimes W.$$
\end{proof}

\subsection{}\label{subsec:defmsetup}  Let $\gn \in \Spec \Hat{R}_X$ be the kernel of $\Hat{R}_X \rightarrow W$ sending the $x_i$ 
to $0,$ and let $\bL_0 = \bL_u\otimes_{\Hat{R}_X} W$ be the corresponding filtered $K3$-crystal over $W.$ 
The underlying $K3$-crystal of $\bL_0$ is canonically identified with $\bL.$

Fix an isomorphism as in Lemma \ref{lem:descentZp}. This allows us to regard 
$\GSO = \GSO(\bL, \bilinear),$ as a group over $\Z_p$. 
Choose an isomorphism 
$j:(\bL_u, \bilinear) \simeq (\bL, \bilinear)\otimes \Hat{R}_X$ inducing the identity mod $\gn.$ 
This gives rise to an isomorphism 
$\GSO(\bL_u, \bilinear) \simeq \GSO(\bL, \bilinear)\otimes \Hat{R}_X,$ and 
the Frobenius maps on $\bL_u$ and $\bL_0$ then are given by $b_u\sigma$ and $b\sigma$ for elements $b_u \in \GSO(\Hat{R}_X[1/p])$ 
and $b \in \GSO(W[1/p])$ such that $b_u$ specializes to $b.$

\begin{proposition}\label{prop:strdef} The isomorphism $(\bL_u, \bilinear) \simeq (\bL, \bilinear)\otimes \Hat{R}_X$ can be chosen so that 
\begin{enumerate}
\item $j$ respects filtrations. 
\item  $b_u = u\cdot b$ where $u \in U^{\opp}(\Hat{R}_X),$ and $U^{\opp}$ is the opposite unipotent of the parabolic 
$P \subset \GSO(\bL_u, \bilinear)$ corresponding to the filtration on $\bL_u.$
\end{enumerate}
Moreover, if these conditions are satisfied, the tautological map $\Spf \Hat{R}_X \xrightarrow{u} \Hat{U}^{\opp}$ is an isomorphism. 
\end{proposition}
\begin{proof} To show (1), we have to show that $j$ can be chosen so that the parabolics $P \subset \GSO(\bL_u, \bilinear)$ and 
$P_0 \subset \GSO(\bL_0, \bilinear)$ corresponding to the filtrations on $\bL_u$ and $\bL_0$ are identified. This follows from that fact that 
any deformation of $P_0$ to a parabolic in  $\GSO(\bL_u, \bilinear)$ is conjugate to the constant deformation \cite{sga3}. 
The choice of such a $j$ is unique up to conjugation by elements of $P(\Hat{R}_X).$ 

Next let $\mu: \mathbb G_m \rightarrow \GSO$ be a cocharacter corresponding to $P.$ We claim that $b_u \in \SO(\Hat{R}_X)\sigma(\mu)(p).$ 
Let 
$$\bL'_u = (b_u\sigma)\mu(p^{-1}) \bL_u = b_u\sigma(\mu)(p^{-1}) \bL_u.$$ 
The conditions on the filtration in a filtered $K3$-crystal imply that $\bL'_u\subset \bL_u,$ 
and that $\mu(p^{-1})$ acts by $p^{-2}$ on $\bilinear.$ Thus for $x, y \in \bL_u$ we have 
$$ \langle (b_u\sigma)\mu(p^{-1}) x, (b_u\sigma)\mu(p^{-1}) y \rangle = p^2 \sigma(\langle \mu(p^{-1}) x, \mu(p^{-1}) y \rangle) = 
\sigma(\langle x, y \rangle). $$
It follows that $\bilinear$ is perfect on $\bL_u',$ and hence $\bL_u' = \bL_u.$ This proves the claim. 

It follows that $w = b_ub^{-1} \in \SO(\Hat{R}_X).$ Since the map $U^{\opp} \rightarrow P_0\backslash \GSO(\bL_0, \bilinear)$ is an open immersion, and $w$ is the identity mod $\gn,$ we can write $w = \lambda\cdot u$ with $u \in U^{\opp}(\Hat{R}_X)$ and 
$\lambda \in P(\Hat{R}_X)$ both reducing to $1$ mod $\gn.$ Conjugating $j$ by an element of $P(\Hat{R}_X)$ 
has the effect of replacing $b_u =  \lambda u b $ by its $\sigma$-conjugate by the same element. 
Now let $m \geq 1,$ and suppose that $j$ can be chosen so that $b_u = b_u(m) = ub$ modulo $\gn^m,$ so that 
$\lambda = \lambda(m) \equiv 1 \mod \gn^m$. Then, 
$$\lambda^{-1}b_u\sigma(\lambda) = ub\sigma(\lambda) = u (b \sigma(\lambda)b^{-1}) b.$$
Since $\lambda \equiv 1 \mod \gn^n,$ $\sigma(\lambda) \equiv 1 \mod \gn^{pm}$, and hence 
 $b_u(m+1) = \lambda^{-1}b_u(m)\sigma(\lambda) \equiv ub \mod \gn^{pm}.$ This shows that $ub$ is the $\sigma$-conjugate of 
 $b_u(1)$ by the convergent product $\dots \lambda(2)\lambda(1),$ so $j$ can be chosen with $b_u = ub.$
 
It remains to show that the map $u:\Spf \Hat{R}_X \xrightarrow{u} \Hat{U}^{\opp}$ is an isomorphism. Given that these are both smooth 20-dimensional formal schemes over $W$, it suffices to prove this modulo the ideal $(p,\fm^2)$. Now let $S = \Hat{R}_X/(p,\fm^2),$ 
and let $\bL_S$ denote filtered $K3$-crystal over $S$ given by $\bL_u|_S.$ After, modifying our chosen isomorphism $j$ by $u,$ 
$\bL_S$ may be identified with 
$$ (\bL\otimes S, u^{-1}ub\sigma(u),u\cdot (\Fil^\bullet \bL)_S, \bilinear) = (\bL\otimes S, b,u\cdot (\Fil^\bullet \bL)_S, \bilinear) $$
where we have used that $\sigma(u) = 1$ in $S.$ 

Work of Nygaard-Ogus \cite[Theorem 5.2, 5.3]{NO} implies that deformations of $X$ to $S$ correspond bijectively to 
\emph{isotropic lifts} of $\Fil^2 \bL \mod {\fm}.$ These lifts are in bijection with points of $U^{\opp}(S)$ which are $1$ modulo $\fm,$ 
and hence with $ \Hat{U}^{\opp}(S).$ This implies that $u$ induces an isomorphism on $S$ points, and hence is an isomorphism.
\end{proof}

\subsection{}\label{S:Spinstructure}
Let $H$ denote the Clifford module associated to $(T,\bilinear')$, and let $\{s_{\alpha,p}\}_{\alpha}\subset H^{\otimes}$ denote tensors whose pointwise stabilizer is the group $\GSpin$ (such tensors exist by \cite[Proposition 1.3.2]{kisinintegral}) of Spinor similitudes associated to $(T,\bilinear')$. 
\begin{defn}\label{Spinstructure}
 A $\GSpin$-structure on a $p$-divisible group $\hdiv/k$ is the data of an isomorphism $\iota: H\otimes W(k) \rightarrow \D(\hdiv) $, such that 
 $\iota(s_{\alpha,p}) \in \D(\hdiv)^{\otimes}$ are Frobenius-invariant. We say that $\iota_1$ and $\iota_2$ are isomorphic $\GSpin$-structures if 
 $\iota_1(s_{\alpha,p}) = \iota_2(s_{\alpha,p})$.
\end{defn}

\subsection{}\label{S:KugaSatakepdiv}
We now recall a Kuga-Satake construction for $K3$-crystals due to Yang \cite[Appendix A]{ziquansymplectic}, which associates a $p$-divisible group with $\GSpin$ structure to a $K3$-crystal. 

Denote by $\bH$ the Clifford module associated to $(\bL, \bilinear)$. The isomorphism $\iota$ induces a canonical isomorphism of Clifford modules 
$H\otimes W(k)\rightarrow \bH$ (where $H$ is as in Section \ref{S:Spinstructure}), which we shall also denote by $\iota$. 

Let $\mu$ be as in the proof of Proposition \ref{prop:strdef}: a $\GSO(\bL)$-valued co-character, defined over $W(k),$ that induces the mod $p$ filtration on $\bL.$ Let $\mu_{\SO}$ denote the co-character of $\SO$ given by $\mu_{\SO}(z) = z^{-1}\mu(z).$ 
The natural map 
\[
\GSpin(\bH)\rightarrow \SO(\bL) 
\]
has kernel $\G_m,$ and there is a unique lift of $\mu_{\SO}$ to a cocharacter $\tilde \mu$ of $\GSpin(\bH)$ such that $\tilde \mu$ acts 
with weights $0$ and $1$ on $\bH.$ Both weights then occur with equal multiplicity. 
We saw in the proof of Proposition \ref{prop:strdef} that $b \in \GSO(W(k))\sigma(\mu)(p).$ 
Hence $p^{-1}b \in \SO(W(k))\sigma(\mu_{\SO})(p),$ and we may lift $p^{-1}b$ to an element $\tilde{b} \in \GSpin(\bH)(W(k))\sigma(\tilde\mu)(p).$
Then $\tilde{b}\bH \subset \bH$, and $\tilde{b}\bH \not\subset p\bH$.
\begin{proposition}[{\cite[Lemma A.7]{ziquansymplectic}}]
There exists a $p$-divisible group $\hdiv$ over $k$ whose Dieudonn\'e module $\D(\hdiv)$ is given by $\bH$, with Frobenius acting as $\tilde{b}\sigma$.
\end{proposition}

\subsection{}\label{S:CMforK3impliesCMforpdiv}
By \cite[Section 4.2.1]{howardpappas}, if $X$ is supersingular, then $\hdiv$ is a supersingular $p$-divisible group.
Since $\tilde b \in \GSpin(W[1/p]),$ the tensors $s_{\alpha,0} = \iota(s_{\alpha,p})\in \bH^{\otimes}$ are stable by Frobenius, and therefore $\hdiv$ is equipped with a canonical $\GSpin$-structure.

We will now use description of the universal deformation space of $X$ to prove that $\hdiv$ admits a CM lift if $X$ does. 
Write $\Hat{R} = \Hat{R}_X.$ The opposite unipotent of $\tilde{\mu}$ in $\GSpin$ is canonically isomorphic to $U^{\opp}$, the opposite unipotent of $\mu_{\SO}$ in $\SO$. Thus, we may regard $u \in U^{\opp}(\Hat{R})$ as an endomorphism of $\bH \otimes \Hat{R}.$

Let $\tilde{\Fil} \subset \bH$ denote the filtration of $\bH$ induced by $\tilde{\mu}.$ By  \cite[Section 1.5]{kisinintegral}, the data 
$(\bH\otimes{\Hat{R}}, \tilde{\Fil}\otimes \Hat{R}, u\cdot(\tilde b\sigma))$ arises from the Dieudonn\'e module of a 
$p$-divisible group $\hdiv_{\Hat{R}}$ over $\Spf \Hat{R},$ which deforms $\hdiv.$ Note that the tensors $s_{\alpha,0} \in \D(\hdiv_{\Hat{R}})^\otimes$ 
are Frobenius invariant and in $\tilde \Fil^0.$

\begin{lemma}\label{lem:wadm} Let $K/W(k)[1/p]$ be a finite extension, $y: \Hat{R} \rightarrow \cO_K,$ a map of $W(k)$-algebras, and 
$X_y$ (resp.~$\hdiv_y)$ the $K3$ surface (resp.~$p$-divisible group) over $\cO_K$ corresponding to $y.$  Let $D(X_y)$ (resp.~$D(\hdiv_y)$) 
denote the weakly admissible module over $K$ associated to $X_y$ (resp.~$\hdiv_y$). Then there is a canonical inclusion 
$$ D(X_y)(1) \subset D(\hdiv_y)^\otimes$$
compatible with filtrations and Frobenius.
\end{lemma}
\begin{proof} The weakly admissible module $D(X_y)$ is constructed using the crystalline and de Rham cohomology of $X_y,$ and 
the isomorphism between them as in \cite[\S 2]{BerthelotOgus}, and similarly for $D(\hdiv_y).$ Let us briefly recall the construction.

Let $\Hat R_y$ denote the $PD$-completion of $\Hat R$ with respect to $\ker (y) + p\Hat R.$ Then 
$\bL_{\Hat R_y} = \bL_u\otimes_{\Hat R} \Hat R_y$ 
is equipped with a Frobenius and filtration. Moreover, there is a unique Frobenius equivariant map 
$\bL = \bL_0 \rightarrow \bL_{\Hat R_y}[1/p]$ which lifts the identity over $\mathfrak n.$ It may be constructed by choosing any lift of the identity 
$s_0,$ and taking the limit $s = \lim_i \varphi^i(s_0) = \lim_i \varphi^i\circ s_0 \circ\varphi^{-i},$ which converges. This allows 
us to identify $\bL\otimes_{W(k)} K$ with $\bL_y[1/p],$ where $\bL_y = \bL_u\otimes_{\Hat R,y} \cO_K.$ If $K_0 \subset K$ denotes the maximal 
unramified subfield, this gives $D(X_y) \cong \bL\otimes_{W(k)} K_0$ the structure of a weakly 
admissible module over $K.$ Note that this identification is not in general given by the identity of $\bL.$
There is an analogous construction starting with $\bH_{\Hat R}$ in place on $\bL_u.$

Now let $\bL_u(1)$ denote the $K3$-crystal with underlying module $\bL_u,$ equipped with the Frobenius given by $p^{-1}b_u\sigma,$ and the filtration given by 
$\Fil^i \bL_u(1) = \Fil^{i+1}\bL_u.$ By construction, there is an inclusion $\bL_u(1) \subset \bH_{\Hat R}^\otimes.$ 
Applying the above construction to both sides, we obtain $D(X_y)(1) \subset D(\hdiv_y)^\otimes,$ as required.
\end{proof}

\begin{proposition}\label{CMforK3impliesCMforpdiv}
Suppose that the $K3$-surface $X$ admits a CM lift. Then so does the associated Kuga-Satake $p$-divisible group $\hdiv$. 
\end{proposition}
\begin{proof} Let $K,K_0$ and $y: \Hat{R} \rightarrow \cO_K,$ be as in the previous lemma, and denote by 
$G(D(X_y)(1))$ and $G(D(\hdiv_y))$ the Tannakian groups of the weakly admissible modules 
$D(X_y)(1)$ and $D(\hdiv_y)^\otimes,$ associated to fiber functor which takes a weakly admissible module to its underlying 
$K_0$-vector space. By construction, we have the following commutative diagram
\[
\begin{xymatrix}
{ G(D(\hdiv_y)) \ar[r]\ar[d] &  G(D(X_y)(1)) \ar[d] \\
\GSpin(\bH_{K_0}) \ar[r] & \SO(\bL_{K_0}(1))
}
\end{xymatrix}
\]
where the vertical maps are the natural inclusions, and the top map is a surjection, which is deduced from Lemma \ref{lem:wadm}. 
Since the bottom map is surjective, with kernel $\G_m,$ it follows that $G(D(\hdiv_y))$ is abelian if and only if $G(D(X_y)(1))$ is abelian. 

Now suppose that $X_y$ is CM. By what we just saw, it suffices to show that $G(D(X_y)(1)),$ or equivalently $G(D(X_y)),$ is abelian. If $\bar K$ denotes an algebraic closure of $K,$ this is equivalent to asking that the $\Gal(\bar K/K)$ action on $H_{\et}^2(X_{y,\bar K},\Q_p) $ is abelian. This 
follows from work of Deligne \cite{deligneK3}. 
\end{proof}

\begin{theorem}\label{thm:K3thm} If $p \geq 5,$ then only finitely many supersingular $K3$-surfaces over $\Fpbar$ admit CM lifts.
\end{theorem}
\begin{proof}
Let $\cS$ denote the set of supersingular $K3$-surfaces over $\Fpbar$ that admit CM lifts. By Proposition \ref{CMforK3impliesCMforpdiv}, the Kuga-Satake $p$-divisible group $\hdiv(X)$ associated to every $X\in \cS$ has a CM lift. Let $\cS'$ denote the set $\{\hdiv(X): \ X\in \cS\}$. Theorem \ref{thm:cmliftpdiv} implies that $\cS'$ contains only finitely\footnote{Corollary \ref{C:introK3thm} does not follow immediately  -- it is a-priori possible that infinitely many non-isomorphic $K3$-surfaces $X$ might yield the same Kuga-Satake $p$-divisible group.} many $\Fpbar$-isomorphism classes of $p$-divisible groups, and therefore there exists a finite field $\F_q$ such that every $\hdiv \in \cS'$ admits a decent model over $\F_q$, (which we will again denote by $\hdiv / \F_q$). By Lemma \ref{lemma:finitetensor} below, every $\GSpin$-structure on $\hdiv$ is defined over $\F_q$. Since the data of a $p$-divisible group $\hdiv$ along with a $\GSpin$-structure uniquely determines the $K3$-crystal, it follows that the $K3$-crystal $\bL(X)$ admits a decent model over $\F_q$ for every $X\in \cS$. 

By Ogus' crystalline Torelli theorem (the main result of \cite{ogustorelli}), it suffices to prove that there are only finitely many isomorphism classes of decent $K3$-crystals defined over any fixed finite field $\F_q$. This is a direct consequence of the discussion in \cite[Section 3, Definition 3.19, Theorem 3.20]{ogusk3} pertaining to characteristic subspaces, and therefore the result follows.
\end{proof}

\begin{lemma}\label{lemma:finitetensor}
Let $\hdiv /\F_q $ denote a decent $p$-divisible group. Then, every $\GSpin$-structure on $\hdiv_{\Fpbar}$ is defined over $\F_q$. 
\end{lemma}
\begin{proof}
Let $\bH = \mathbb{D}(\hdiv),$  and $\iota: H\otimes_WW(\Fpbar)\rightarrow \bH_{W(\Fpbar)}$ a $\GSpin$-structure on $\hdiv$. Let $s_{\alpha,0}  = \iota(s_{\alpha,p})\in \bH^{\otimes}_{W(\Fpbar)}.$  Write $\varphi$ for the Frobenius on $\bH.$ 
By definition, the $s_{\alpha,0}$ are Frobenius-invariant tensors.    
The assumption that $\hdiv$ is decent implies that the $q$-power Frobenius acts on $\bH = \mathbb{D}(\hdiv)$ as the scalar $q^{1/2}$. 
As $s_{\alpha,0}$ is $\varphi\otimes\sigma$-invariant, it lies in the slope $0$ part of 
$\bH_{W(\Fpbar)}[1/p].$ Hence, if $q = p^r,$ then $\varphi^r$ acts trivially on $s_{\alpha,0},$ and so $\sigma^r$ acts trivally on $s_{\alpha,0}.$ 
Hence, $s_{\alpha,0} \in \bH^\otimes.$ 

Now consider the $W(\F_q)$-scheme that represents the functor, which assigns to an $W(\F_q)$-algebra $R$ the set of 
isomorphisms $H\otimes R \rightarrow \bH \otimes R$ that sends $s_{\alpha,p}$ to $s_{\alpha,0}.$ This scheme is a $\GSpin_W$-torsor as it has a point (given by the isomorphism $\iota$) defined over $W(\Fpbar)$. As $\GSpin$ is a connected reductive group, Lang's lemma  yields that the torsor must be trivial over $W(\F_q)$, and hence there exists an isomorphism $\iota_q: H\otimes W(\F_q) \rightarrow \bH$ that respects tensors. It follows that the $\GSpin$-structure is indeed defined over $\F_q$, as claimed 
\footnote{We note that the same proof goes through to prove that every $G$-structure on $\hdiv_{\Fpbar}$ descends to $\hdiv$ whenever $G/\Z_p$ is a connected reductive group. Without the connectedness assumption, the same result holds upto replacing $\F_q$ with a finite extension where the degree of the extension depends only on the component set of $G$.}. 
\end{proof}

\bibliography{HeckeReduction.bib}

\begin{bibdiv}
\begin{biblist}

\bib{BerthelotOgus}{article}{
      author={Berthelot, P.},
      author={Ogus, A.},
       title={{$F$}-isocrystals and de {R}ham cohomology. {I}},
        date={1983},
        ISSN={0020-9910},
     journal={Invent. Math.},
      volume={72},
      number={2},
       pages={159\ndash 199},
         url={https://doi-org.ezproxy.library.wisc.edu/10.1007/BF01389319},
      review={\MR{700767}},
}

\bib{BhattScholze}{article}{
      author={Bhatt, Bhargav},
      author={Scholze, Peter},
       title={Prisms and prismatic cohomology},
        date={2019},
     journal={Preprint, available at https://arxiv.org/abs/1905.08229},
}

\bib{conradcmlifting}{book}{
      author={Chai, Ching-Li},
      author={Conrad, Brian},
      author={Oort, Frans},
       title={Complex multiplication and lifting problems},
      series={Mathematical Surveys and Monographs},
   publisher={American Mathematical Society, Providence, RI},
        date={2014},
      volume={195},
      review={\MR{3137398}},
}

\bib{deligneK3}{article}{
      author={Deligne, Pierre},
       title={La conjecture de {W}eil pour les surfaces {$K3$}},
        date={1972},
     journal={Invent. Math.},
      volume={15},
       pages={206\ndash 226},
         url={https://doi-org.ezproxy.library.wisc.edu/10.1007/BF01404126},
      review={\MR{296076}},
}

\bib{sga3}{book}{
      editor={Gille, Philippe},
      editor={Polo, Patrick},
       title={Sch\'{e}mas en groupes ({SGA} 3). {T}ome {III}. {S}tructure des
  sch\'{e}mas en groupes r\'{e}ductifs},
      series={Documents Math\'{e}matiques (Paris) [Mathematical Documents
  (Paris)]},
   publisher={Soci\'{e}t\'{e} Math\'{e}matique de France, Paris},
        date={2011},
      volume={8},
        ISBN={978-2-85629-324-9},
        note={S\'{e}minaire de G\'{e}om\'{e}trie Alg\'{e}brique du Bois Marie
  1962--64. [Algebraic Geometry Seminar of Bois Marie 1962--64], A seminar
  directed by M. Demazure and A. Grothendieck with the collaboration of M.
  Artin, J.-E. Bertin, P. Gabriel, M. Raynaud and J-P. Serre, Revised and
  annotated edition of the 1970 French original},
      review={\MR{2867622}},
}

\bib{howardpappas}{article}{
      author={Howard, Benjamin},
      author={Pappas, Georgios},
       title={Rapoport--zink spaces for spinor groups},
        date={2017},
     journal={Compositio Mathematica},
      volume={153},
      number={5},
       pages={1050\ndash 1118},
}

\bib{itocmlift}{article}{
      author={Ito, Kazuhiro},
      author={Ito, Tetsushi},
      author={Koshikawa, Teruhisa},
       title={Cm liftings of k3 surfaces over finite fields and their
  applications to the tate conjecture},
        date={2018},
     journal={arXiv preprint arXiv:1809.09604},
}

\bib{kisinfcrystal}{incollection}{
      author={Kisin, Mark},
       title={Crystalline representations and {$F$}-crystals},
        date={2006},
   booktitle={Algebraic geometry and number theory},
      series={Progr. Math.},
      volume={253},
   publisher={Birkh\"{a}user Boston, Boston, MA},
       pages={459\ndash 496},
      review={\MR{2263197}},
}

\bib{kisinintegral}{article}{
      author={Kisin, Mark},
       title={Integral models for {S}himura varieties of abelian type},
        date={2010},
     journal={J. Amer. Math. Soc.},
      volume={23},
      number={4},
       pages={967\ndash 1012},
      review={\MR{2669706}},
}

\bib{kottwitz}{article}{
      author={Kottwitz, Robert~E.},
       title={Isocrystals with additional structure},
        date={1985},
        ISSN={0010-437X},
     journal={Compositio Math.},
      volume={56},
      number={2},
       pages={201\ndash 220},
      review={\MR{809866}},
}

\bib{moretbailly}{incollection}{
      author={Moret-Bailly, Laurent},
       title={Familles de courbes et de vari\'{e}t\'{e}s ab\'{e}liennes sur
  {${\Bbb P}^1$}. {I}. {D}escente des polarisations},
        date={1981},
       pages={109\ndash 124},
        note={Seminar on Pencils of Curves of Genus at Least Two},
      review={\MR{3618575}},
}

\bib{keerthi}{article}{
      author={Madapusi~Pera, Keerthi},
       title={The {T}ate conjecture for {K}3 surfaces in odd characteristic},
        date={2015},
     journal={Invent. Math.},
      volume={201},
      number={2},
       pages={625\ndash 668},
      review={\MR{3370622}},
}

\bib{NO}{article}{
      author={Nygaard, Niels},
      author={Ogus, Arthur},
       title={Tate's conjecture for {$K3$} surfaces of finite height},
        date={1985},
     journal={Ann. of Math. (2)},
      volume={122},
      number={3},
       pages={461\ndash 507},
         url={https://doi-org.ezproxy.library.wisc.edu/10.2307/1971327},
      review={\MR{819555}},
}

\bib{ogusk3}{incollection}{
      author={Ogus, Arthur},
       title={Supersingular {$K3$} crystals},
        date={1979},
   booktitle={Journ\'{e}es de {G}\'{e}om\'{e}trie {A}lg\'{e}brique de {R}ennes
  ({R}ennes, 1978), {V}ol. {II}},
      series={Ast\'{e}risque},
      volume={64},
   publisher={Soc. Math. France, Paris},
       pages={3\ndash 86},
      review={\MR{563467}},
}

\bib{ogustorelli}{incollection}{
      author={Ogus, Arthur},
       title={A crystalline torelli theorem for supersingular k3 surfaces},
        date={1983},
   booktitle={Arithmetic and geometry},
   publisher={Springer},
       pages={361\ndash 394},
}

\bib{oortfoliation}{article}{
      author={Oort, Frans},
       title={Foliations in moduli spaces of abelian varieties},
        date={2004},
     journal={J. Amer. Math. Soc.},
      volume={17},
      number={2},
       pages={267\ndash 296},
      review={\MR{2051612}},
}

\bib{oortdimension}{incollection}{
      author={Oort, Frans},
       title={Foliations in moduli spaces of abelian varieties and dimension of
  leaves},
        date={2009},
   booktitle={Algebra, arithmetic, and geometry: in honor of {Y}u. {I}.
  {M}anin. {V}ol. {II}},
      series={Progr. Math.},
      volume={270},
   publisher={Birkh\"{a}user Boston, Boston, MA},
       pages={465\ndash 501},
      review={\MR{2641199}},
}

\bib{oortjag}{article}{
      author={Oort, Frans},
       title={C{M}-liftings of abelian varieties},
        date={1992},
        ISSN={1056-3911},
     journal={J. Algebraic Geom.},
      volume={1},
      number={1},
       pages={131\ndash 146},
      review={\MR{1129842}},
}

\bib{rapoportzink}{book}{
      author={Rapoport, M.},
      author={Zink, Th.},
       title={Period spaces for {$p$}-divisible groups},
      series={Annals of Mathematics Studies},
   publisher={Princeton University Press, Princeton, NJ},
        date={1996},
      volume={141},
      review={\MR{1393439}},
}

\bib{sen}{article}{
      author={Sen, Shankar},
       title={Lie algebras of {G}alois groups arising from {H}odge-{T}ate
  modules},
        date={1973},
     journal={Ann. of Math. (2)},
      volume={97},
       pages={160\ndash 170},
      review={\MR{0314853}},
}

\bib{serre1}{incollection}{
      author={Serre, Jean-Pierre},
       title={Groupes alg\'{e}briques associ\'{e}s aux modules de
  {H}odge-{T}ate},
        date={1979},
   booktitle={Journ\'{e}es de {G}\'{e}om\'{e}trie {A}lg\'{e}brique de {R}ennes.
  ({R}ennes, 1978), {V}ol. {III}},
      series={Ast\'{e}risque},
      volume={65},
   publisher={Soc. Math. France, Paris},
       pages={155\ndash 188},
      review={\MR{563476}},
}

\bib{ziquansymplectic}{article}{
      author={Yang, Ziquan},
       title={On irreducible symplectic varieties of {$ \mathrm{K3}^{[n]}
  $}-type in positive characteristic},
        date={2019},
     journal={arXiv preprint arXiv:1911.05653},
}

\bib{zarhin}{article}{
      author={Zarhin, Yurii~Gennad'evich},
       title={Endomorphisms of abelian varieties and points of finite order in
  characteristic p},
        date={1977},
     journal={Matematicheskie Zametki},
      volume={21},
      number={6},
       pages={737\ndash 744},
}

\end{biblist}
\end{bibdiv}

\end{document}